\documentclass[a4paper,11pt]{amsart}

\usepackage{amsmath,amssymb,amsthm}
\usepackage{graphicx}
\usepackage[all]{xy}
\usepackage{wrapfig}
\usepackage{pxfonts}

\DeclareMathAlphabet{\mathbbold}{U}{bbold}{m}{n}

\def\k{\mathbbold{k}}

\DeclareSymbolFont{rsfscript}{OMS}{rsfs}{m}{n}
\DeclareSymbolFontAlphabet{\mathrsfs}{rsfscript}

\DeclareFontFamily{OMS}{rsfs}{\skewchar\font'177}
\DeclareFontShape{OMS}{rsfs}{m}{n}{%
      <5> rsfs5
      <6> <7> rsfs7
      <8> <9> <10> rsfs10
      <10.95> <12> <14.4> <17.28> <20.74> <24.88> rsfs10
      }{}

\def\calM{\mathrsfs{M}}

\DeclareMathOperator{\Cl}{Cl}
\DeclareMathOperator{\cl}{cl}

\DeclareMathOperator{\st}{st}

\makeatletter

\theoremstyle{plain}

\newtheorem {theorem}{Theorem}
\newtheorem {lemma}{Lemma}
\newtheorem {corollary}{Corollary}

\newtheorem {proposition}{Proposition}

\theoremstyle{definition}

\newtheorem {definition}{Definition}

\newtheorem {example}{Example}

\begin{document}

\title{Anick-type resolutions\\ and consecutive pattern avoidance}
\author{Vladimir Dotsenko}
\address{Dublin Institute for Advanced Studies, 10 Burlington Road, Dublin 4, Ireland and School of Mathematics, Trinity College, Dublin 2, Ireland}
\email{vdots@maths.tcd.ie}
\author{Anton Khoroshkin}
\address{Departement Matematik, ETH, R\"amistrasse 101, 8092 Zurich, Switzerland and 
ITEP, Bolshaya Cheremushkinskaya 25, 117259, Moscow, Russia}
\email{anton.khoroshkin@math.ethz.ch}

\thanks{The first author's research was supported by the grant RFBR-CNRS-07-01-92214 and by an IRCSET research fellowship. The second author's research was supported by grants
RFBR-10-01-00836, NSh-65290.2010.2, RFBR-CNRS-07-01-92214, and by a ETH research fellowship.
}

\begin{abstract}
For permutations avoiding consecutive patterns from a given set, we present a combinatorial formula for the multiplicative inverse of the corresponding exponential generating function. The formula comes from homological algebra considerations in the same sense as the corresponding inversion formula for avoiding word patterns comes from the well known Anick's resolution. 
\end{abstract}

\maketitle

\section{Introduction}

The purpose of this paper is to present a formula for the multiplicative inverse for the formal series enumerating permutations avoiding the given set of consecutive patterns. There are various formulas of that sort, one based on a version of inclusion--exclusion principle, namely the cluster method of Goulden and Jackson~\cite{GJ2,NZ}, and another more recent and much more compact, based on the symmetric functions method of Mendes and Remmel~\cite{MR}. Our formula should be thought of as a mixture of these two: on one hand, it is based on combinatorial data somewhat similar to Goulden--Jackson clusters, on the other hand, it takes care of most cancellations which more resembles what happens in~\cite{MR}.  

In the case of pattern avoidance in words, a similar result was obtained by Anick~\cite{Anick}; however, the emphasis of his paper was on applications to homological algebra, and it never attracted attention of specialists in enumerative combinatorics. Consequently, a rather straightforward generalization of his method to the case of consecutive patterns in permutations has never been discovered. We present such a generalisation in this paper. Our intuition here comes from homological algebra as well; our construction is based on free resolutions of Anick type for shuffle algebras~\cite{Ronco} defined by generators and relations. This approach extends without any changes to the case of coloured permutations avoiding the given set of consecutive patterns~\cite{Ma}. For the original Goulden--Jackson formula, a homological proof can also be obtained, using constructions in the spirit of~\cite{DKRes}; we do not intend to discuss it in detail. 

This paper is organized as follows. In Section~\ref{cluster} we briefly recall the Goulden--Jackson cluster method, define ``chains'' (which give an improved version of Goulden--Jackson clusters), and prove the inversion formula. To make most of the text accessible to the general mathematical audience, we chose to present the proof in the most elementary way, and use a sign-reversing involution instead of talking about boundary maps and chain homotopies.

In Section~\ref{examples}, we consider various applications. Even staying within the cluster method, it is possible to solve some problems on consecutive pattern avoidance. In particular, we deduce a result in theory of pattern avoidance that seems to be new: for a permutation~$\tau$ of length~$k$ without self-overlaps, the number of permutations of given length with the given number of occurrences of~$\tau$ depends only on $k$, $\tau(1)$, and $\tau(k)$. This was formulated as a conjecture by Sergi Elizalde~\cite{Elizalde}.\footnote{While preparing this paper, we learned that this conjecture is proved in an upcoming paper of Jeffrey Remmel, based on methods developed in~\cite{MR}. We wish to thank Sergey Kitaev for informing us of that.} We also show how our method applies to patterns of length~$4$ (obtaining some formulas that seem to be new), simultaneous avoidance of some patterns of length~$3$, and a well known result on rises in permutations..

In Appendix, we briefly explain the homological algebra behind the story, putting our proofs in the context of homological algebra for shuffle algebras~\cite{Ronco} and ideals of the associative operad~\cite{Lat1,Lat2}.

The authors wish to thank Sergi Elizalde and Sergey Kitaev for their remarks. The work on this paper started when the second author was visiting Dublin Institute for Advanced Studies; he expresses his gratitude to all the staff there for their hospitality. 

\section{Clusters and chains}\label{cluster}

\subsection{Consecutive pattern avoidance}

Let us recall some definitions and notation. A permutation of length $n$ is a sequence containing each of the numbers $1$,\ldots, $n$ exactly once. To every sequence $s$ of length $k$ consisting of $k$ distinct numbers, we assign a permutation $\st(s)$ of length $k$ called the standardization of~$s$; it is uniquely determined by the condition that $s_i< s_j$ if and only if $\st(s)_i<\st(s)_j$, for example, $\st(153)=(132)$. In other words, $\st(s)$ is a permutation whose relative order of entries is the same as that of~$s$. We say that a permutation $\sigma$ avoids the given permutation $\tau$ as a consecutive pattern if for each $i<j$ we have $\st(\sigma_i\sigma_{i+1}\ldots\sigma_j)\ne\tau$, otherwise we say that $\sigma$ contains $\tau$ as a consecutive pattern. Throughout this paper, we only deal with consecutive patterns, so the word ``consecutive'' will be omitted. For historical information on pattern avoidance in general and the state-of-art for consecutive patterns, we refer the reader to~\cite{KitHist,Stein}. 

The central question arising in the theory of pattern avoidance is that of enumeration of permutations of given length that avoid the given set of forbidden patterns~$P$ or, more generally, contain the given number of occurrences of patterns from~$P$. This question naturally leads to the following equivalence relations. Two sets of patterns $P$ and $P'$ are said to be Wilf equivalent (notation: $P\simeq_W P'$) if for every $n$, the number of $P$-avoiding permutations of length~$n$ is equal to the number of $P'$-avoiding permutations of length~$n$. This notion (in the case of one pattern) is due to Wilf~\cite{Wilf}. More generally, $P$ and $P'$ are said to be equivalent (notation: $P\simeq P'$) if for every $n$ and every $0\le k<n$, the number of permutations of length~$n$ with~$k$ occurrences of patterns from $P$ is equal to the number of permutations of length~$n$ with $k$ occurrences of patterns from~$P'$.
 
While studying the equivalence classes of patterns, sometimes it is possible to replace the set of forbidden patterns by an equivalent one with less patterns in it. 
Namely, we have a partial ordering on the set of all permutations (of all possible lengths), namely, $\tau<\sigma$ if $\sigma$ contains $\tau$ as a consecutive pattern. Given a set $P$ of ``forbidden'' patterns, to enumerate the permutations avoiding all patterns from~$P$, we may assume that $P$ is an antichain with respect to this partial ordering. Indeed, ignoring all patterns from~$P$ that contain a smaller forbidden subpattern does not change the set of $P$-avoiding permutations. Therefore, throughout the paper we shall assume that forbidden patterns do indeed form an antichain.

\subsection{Cluster method}

The cluster method of Goulden and Jackson~\cite{GJ2,NZ} is a powerful method of enumeration of words and permutations according to the number of consecutive occurences of certain patterns. Informally, a cluster is a way to link together several patterns from the given set. More precisely in the case of permutations $q$-clusters relative to the given pattern set~$P$ are triples $(\sigma,\pi_1,\ldots,\pi_q,i_1,\ldots,i_q)$ such that
\begin{itemize}
\item[-] for every $k=1,\ldots,q$, $\st(\sigma_{i_1},\ldots,\sigma_{i_k+r_k-1})=\pi_k$, where $r_k$ is the length of~$\pi_k$ ($i_k$ marks the occurrence of the $k^\text{th}$ pattern $\pi_k$);
\item[-] $l_{k+1}>l_k$ (patterns are listed from the left to the right) and $l_{k+1}<l_k+r_k$ (adjacent patterns are linked);
\item[-] $l_1=1$, and the length of $\sigma$ is equal to $l_q+r_q-1$ ($\sigma$ is completely covered by patterns $\pi_1,\ldots,\pi_q$).
\end{itemize}

Let us denote by~$\pi_{n,k}$ the number of permutations of length~$n$ with exactly $k$ occurrences of patterns from~$P$, and by
$\cl_{n,q}$~--- the number of $q$-clusters where the permutation~$\sigma$ has length~$n$. We also consider the generating functions 
$\Pi(x,t)=\sum_{n,k}\pi_{n,k}\frac{x^n}{n!}t^k,$ and $\Cl(x,t)=\sum_{n,q}\cl_{n,q}\frac{x^n}{n!}t^q$. 
The following enumeration theorem is an immediate consequence of the cluster method.
\begin{theorem}[\cite{GJ2}]
We have 
\begin{equation}\label{enum_cluster}
\Pi(x,t)=\frac{1}{1-x-\Cl(x,t-1)}. 
\end{equation}
Consequently, the exponential generating function for permutations avoiding patterns from the set $P$ is  
\begin{equation}\label{avoid_cluster}
\frac{1}{1-x-\Cl(x,-1)}.
\end{equation}
\end{theorem}

\subsection{Chains and series inversion}\label{chains}

In this section, we show how to improve the cluster method inversion formula for pattern avoidance~\eqref{avoid_cluster} for a general pattern set~$P$. Basically, for patterns without self-overlaps, there is nothing to improve: no cancellations happen in the formulas written above. However, in general a permutation $\sigma$ can occur in several different $q$-clusters for different~$q$, which will result in cancellations in $Cl(x,-1)$. We shall explain what combinatorial objects correspond to coefficients after these obvious cancellations. We still assume that $P$ is an antichain, that is, patterns from~$P$ are not contained in one another.

The main combinatorial concepts we need are $q$-chains and their tails. They are defined inductively as follows:
\begin{itemize}
\item[-] empty permutation is a $0$-chain, it coincides with its tail;
\item[-] the only permutation of one element is a $1$-chain, it also coincides with its tail;
\item[-] each $q$-chain is a permutation $\sigma$ equal to the concatenation $\sigma'\tau$, where $\tau$ is the tail of $\sigma$, and $\st(\sigma')$ is a $(q-1)$-chain;
\item[-] if we denote by~$\tau'$ the tail of $\sigma'$ in the above decomposition, there exists a ``factorization'' $\tau'=\alpha\beta$  with $\st(\beta\tau)\in P$, 
and $\beta\tau$ is the only occurrence of a pattern from~$P$ in $\tau'\tau$.
\end{itemize}

The way we define the chains here is slightly different from the original approach of Anick~\cite{Anick}; the reader familiar with the excellent textbook of Ufnarovskii~\cite{Ufn} will rather notice similarities with the approach to Anick's resolution adopted there.

Basically, $(q+1)$-chains are $q$-clusters with additional restrictions: only neighbours are linked, the first $q$ patterns form an $q$-chain, and no proper beginning forms an $(q+1)$-chain (to be precise, one should apply standardization for the last two properties to make sense).

Let us give some examples clarifying the notion of a chain. For example, if $P=\{12\}$, the only $n$-chain for each $n$ is $12\ldots (n+1)$, while if $P=\{123\}$, we can easily see that $123$ is the only $1$-chain, and $1234$ is the only $2$-chain, but $12345$ is not a $2$-chain because it starts from a $2$-chain $1234$, and is not a $3$-chain because in the only tiling of this permutation by three copies of our pattern its first and third occurences overlap: 
 $$
\underbrace{1\quad 
 2\quad \makebox[0pt][l]{$\overbrace{\phantom{2\quad 3\quad 4}}$}3}\quad 4\quad 5.
 $$

\begin{lemma}
If $\sigma$ is an $n$-chain, the way to link patterns from~$P$ to one another to form~$\sigma$ is unique.
\end{lemma}

\begin{proof}
Assume that there are two ways to link $n$ patterns to form~$\sigma$. Obviously, for each $m$, the endpoints of the $m^\text{th}$ patterns in these two linkages should coincide, otherwise we shall find an $m$-chain whose proper beginning is an $m$-chain as well. Once we know that the endpoints of $m^\text{th}$ patterns are the same, the beginnings have to be the same because $P$ is assumed to be an antichain (and so patterns from~$P$ cannot be contained in one another).
\end{proof}

Let us denote by $c_{n,k}$ the number of $k$-chains in~$\Sigma_n$; we put 
 $$c_n=\sum_{k}(-1)^kc_{n,k}.$$  
We also denote by $a_n$ the number of permutations in~$\Sigma_n$ avoiding all patterns from~$P$. Our main result here is that the corresponding exponential generation functions are multiplicative inverse to one another. Namely, let
 $$
A(t)=\sum_{n\ge0}a_n\frac{t^n}{n!}, \quad C(t)= \sum_{k\ge0}c_n\frac{t^n}{n!}.
 $$
\begin{theorem}
We have 
\begin{equation}\label{master}
A(t)C(t)=1.
\end{equation}
\end{theorem}

\begin{proof}
Let $M_{n,k}$ be the set of all permutations $\gamma\in\Sigma_n$ decomposed as a concatenation $\sigma\lambda$ where $\st(\sigma)$ is a $k$-chain, and $\lambda$ is avoiding all patterns from~$P$. We put $M=\bigsqcup_{n,k}M_{n,k}$. We shall distinguish different factorisations of the same $\gamma$, so we adopt the notation $\sigma\mid\lambda$ with a bar between the factors for elements of~$M$. 

We define an involution $I\colon M\to M$ as follows. We have $\sigma=\sigma'\tau$, where $\tau$ is the tail of $\sigma$, and there are two possibilities: either $\tau\lambda$ is avoiding all patterns from $P$, or $\tau\lambda=\tau'\pi\lambda'$, where $\st(\pi)\in P$, and $\pi$ is the leftmost occurrence of a pattern from~$P$ in $\tau\lambda$. In the first case, we put $I(\sigma\mid\lambda)=\sigma'\mid\tau\lambda$. In the second case, we put $I(\sigma\mid\lambda)=\sigma'\tau'\pi\mid\lambda'$. Informally, if it is possible to move the tail of $\sigma$ through the bar without creating occurrences of forbidden patterns, we do it, and otherwise we extend $\sigma$ to a $(k+1)$-chain using some elements in the beginning of $\lambda$. It is obvious that $I$ is an involution. Also, it is clear that it changes the parity of the parameter $k$. Thus, for each $n\ge1$, the $n^\text{th}$ coefficient of $A(t)C(t)$, that is 
 $$
\sum_n\binom{n}{p}c_pa_{n-p}=\sum_{n,k}(-1)^k\binom{n}{p}c_{p,k}a_{n-p},
 $$
is equal to~$0$ (which implies the formula~\eqref{master}). Indeed, the number $\binom{n}{p}c_{p,k}a_{n-p}$ is the number of elements in $\calM_{n,k}$ for which $\st(\sigma)$ is a permutation of~$p$ elements. The above sum of these number with appropriate signs compute the difference between the number of such elements where $k$ is even and the number of such $\gamma$'s where~$k$ is odd. However, our involution establishes a bijection between these two subsets, so this difference is equal to~$0$, as required. 
\end{proof}

\section{Examples}\label{examples}

Before moving on to particular results, let us state a general remark. Our results suggest that the class of power series that contains all inverses of pattern avoidance enumerators is related to some nice combinatorics. Results of Elizalde and Noy~\cite{EN} that we re-prove in Section~\ref{app-cluster} describe some of these series as solutions to particular differential equations. Our formulas for other cases we considered can be rewritten as more complicated functional equations. What can be said about other series of that sort? From our approach it is always possible to derive recurrence relations, provided some additional statistics are taken into account. However, so far we were not able to describe a reasonable class of series that cover all of these. For example, a wild guess is that all these series satisfy algebraic differential equations, that is, if $f(x)$ is such a series, then $P(x,f(x),f'(x),\ldots,f^{(n)}(x))=0$ for some polynomial~$P$. 

\subsection{Applications of the cluster method}\label{app-cluster}

\subsubsection{Patterns without self-overlaps, unlabelled clusters, and posets}

There are fairly many situations where clusters are the same as chains. In this section, we consider some of these situations, namely, the case of an arbitrary pattern without self-overlaps, the case of one pattern of length~$4$, and the pair of patterns $(132,231)$.

\begin{definition}
A pattern $\tau\in\Sigma_m$ is said to have no self-overlaps if every permutation of length at most $2m-2$ has at most one occurrence of~$\tau$. (Clearly, there always exist permutations of length~$2m-1$ with two occurrences of $\tau$.)
\end{definition}

For example, the pattern $132$ is of that form: clearly, we can only link it with itself using the last entry. A more general example studied in~\cite{EN} is $12\ldots a\ \tau\ (a+1)\in\Sigma_n$, where $a+1<n$, and $\tau$ is an arbitrary permutation of the numbers $a+2,\ldots,n$. 

For a pattern~$\tau$ without self-overlaps, there exists a simple way to reformulate the enumeration problem for clusters in terms of total orderings on posets. The first author used this method in~\cite{DVJ} in a similar setting, dealing with tree monomials in the free shuffle operad. Let form an ``unlabelled cluster'' of the shape that we expect, namely, replace temporarily each entry in the expected cluster by the symbol~$\bullet$ (a bullet). For example, for the pattern~$1243$ we get
 $$
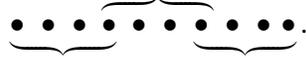
\underbrace{\bullet\,\,\bullet\,\,\bullet\,\,\,
\makebox[0pt][l]{$\overbrace{\phantom{\bullet\,\,\bullet\,\,\bullet\,\,\bullet\,\,}}$}\bullet}\,\,\bullet\,\,\bullet\,\,
\underbrace{\bullet\,\,\bullet\,\,\bullet\,\,\bullet}.
 $$ 
For such an unlabelled cluster $\gamma$, let us define a partial ordering on the set of bullets of $\gamma$ as follows: for each $j$, we equip the $j^\text{th}$ bullet pattern with a total ordering inherited from the real pattern~$\tau$. Let us denote by $\Pi_\gamma$ the thus defined poset.

\begin{example}
Let us take the bullet pattern above, and replace bullet by letters, to make it easier to distinguish between the different bullets: 
 $$
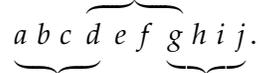
\underbrace{a\,\,b\,\,c\,\,\,
\makebox[0pt][l]{$\overbrace{\phantom{d\,\,e\,\,f\,\,g\,\,}}$}d}\,\,e\,\,f\,\,
\underbrace{g\,\,h\,\,i\,\,j}.
 $$ 
Then the orderings inherited from~$1243$ are $a<b<d<c$, $d<e<g<f$, and $g<h<i<j$, so we obtain the poset
 $$ 
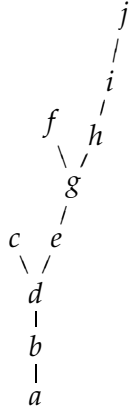
\xygraph{!{<0mm,0mm>;<1.5mm,0mm>:<0mm,1.5mm>::}
!{(47.5263,31.7229)*+{a}}="0"
!{(47.5263,36.4093)*+{b}}="1"
!{(47.5263,41.0957)*+{d}}="2"
!{(45.6784,45.7822)*+{c}}="3"
!{(49.3742,45.7822)*+{e}}="4"
!{(50.8250,50.4686)*+{g}}="5"
!{(48.8709,56.1550)*+{f}}="6"
!{(52.8792,55.1550)*+{h}}="7"
!{(54.1106,59.8415)*+{i}}="8"
!{(55.2824,65.6279)*+{j}}="9"
"0"-"1"
"1"-"2"
"2"-"3"
"2"-"4"
"4"-"5"
"5"-"6"
"5"-"7"
"7"-"8"
"8"-"9"
}
 $$
(the covering relation of the poset is, as usual, represented by edges; $v$ is covered by $w$ if $w$ is the top vertex of the corresponding edge).
\end{example}

The following proposition is obvious. 
\begin{proposition}
The set of $k$-clusters for $P=\{\tau\}$, where $\tau$ has no self-overlaps is in one-to-one correspondence with the set of all total orderings on posets $\Pi_\gamma$ for unlabelled $k$-clusters $\gamma$. 
\end{proposition}

Now we shall see how this approach can be applied in some cases. 

\subsubsection{Case of the pattern~$12\ldots a\ \tau\ (a+1)$}

Let $a<m$, and let $12\ldots a\ \tau\ (a+1)$ be a permutation in $\Sigma_{m+1}$ which starts from the rise $1,2,\ldots,a$, followed by some permutation~$\tau$ of $(a+2),\ldots,m+1$, followed by the number~$(a+1)$. Clearly, this pattern has no self-overlaps, so to enumerate clusters we may count total orderings of posets. Note that every $k$-cluster for $k\ge1$ is of length $k(m+1)-(k-1)=km+1$. 

\begin{proposition}
For $P=\{12\ldots a\ \tau\ (a+1)\}$, the number of~$k$-clusters is equal to
 $$
\prod_{j=1}^k\binom{jm-a}{m-a}.
 $$
\end{proposition}

\begin{proof}
This proof serves us as a starting example of how to use posets to study clusters. The poset which we need to enumerate $k$-clusters in this case looks like a tree of height $m+1$ with the only branch growing on the height~$a+1$, this branch being of length~$m+1$ and having a smaller branch growing at the distance $a+1$ from the starting point etc. (An example of such a poset for the case of the permutation~$1243$ with $a=2$, $m=3$ is given above.) To extend such a partial ordering to a total ordering, we should make the lowest $a+1$ elements for such a tree the smallest elements $1,2,\ldots,a+1$ of the resulting ordering. Then, there are $\binom{km-a}{m-a}$ ways to choose $m-a$ remaining elements forming the stem of our tree, and we are left with the same question for a smaller tree, where we may proceed by induction.
\end{proof}

\begin{corollary}[see \cite{EN,KitPOP} for $t=0$]\label{1non}
For $a<m$, the multiplicative inverse of the generating function $\Pi$ counting occurrences of $12\ldots a\ \tau\ (a+1)\in S_{m+1}$ is given by the formula
\begin{equation}\label{EN}
1-x-\sum_{k\ge1}\frac{(t-1)^{k} x^{km+1}}{(km+1)!}\prod_{j=1}^k \binom{jm-a}{m-a}. 
\end{equation}
In particular, all these patterns, for different~$\tau$, are equivalent to each other.
\end{corollary}

Except for the case of the pattern $123\simeq 321$, this covers all patterns of length~$3$, because $132\simeq 312\simeq 231\simeq 213$. We shall deal with the pattern~$123$ and, more generally, $12\ldots a$, in section~\ref{app-chain}.

\subsubsection{Case of one arbitrary pattern without self-overlaps}

Generalizing the previous result, let us consider an arbitrary pattern $\tau\in S_{m+1}$ without self-overlaps. Then every $k$-cluster for $k\ge1$ is of length~$km+1$. The following result was conjectured in~\cite{Elizalde}, where it was proved in some particular cases. Another proof in the general case was, as we were informed by Sergey Kitaev, obtained by  Jeffrey Remmel, the proof being based on methods developed in~\cite{MR}.

\begin{theorem}\label{wilf}
For a pattern $\tau\in S_{m+1}$ without nontrivial self-overlaps, the number of permutations of length~$n$ with $k$ occurrences of $\tau$ depends only on $m$, $\tau(1)$, and $\tau(m+1)$. In other words, two non-self-overlapping permutations in $S_{m+1}$ are equivalent if their first and last entries are the same.
\end{theorem}

\begin{proof}
This result is also very easy to derive using posets. To make formulas compact, let us put $a=\tau(1)-1$ and $b=\tau(m+1)-1$.
The poset whose total orderings enumerate $k$-clusters is obtained from $k$ totally ordered sets of cardinality~$m+1$ as follows: the element $a+1$ of the second set is identified with the element~$b+1$ of the first set, the element $a+1$ of the third set is identified with the element $b+1$ of the second set, etc. Clearly, this poset depends only on $m$, $a$, and $b$. 

The actual number of $k$-clusters in this case can be computed as follows. Let us denote by $f_k(p)$ the number of $k$-clusters $\sigma$ with $\sigma(1)=p+1$. Then it is easy to see that the following recurrence relation holds (here we assume, without the loss of generality, that $a<b$):
\begin{equation}\label{recur}
f_k(p)=\sum_{q}\binom{p}{a}\binom{km-q}{m-b}\binom{q-p-1}{b-a-1}f_{k-1}(q-b). 
\end{equation}
Indeed, if we denote $q+1=\sigma(m+1)$, there are $\binom{p}{a}$ ways to choose elements less than~$\sigma(1)$ in the first pattern in the cluster, $\binom{km-q}{m-b}$ ways to choose elements greater than $\sigma(m+1)$ there, $\binom{q-p-1}{b-a-1}$ to fill the space between these elements, and $f_{k-1}(q-b)$ ways to choose the remaining $(k-1)$-cluster. 
\end{proof}

\begin{example}
Theorem~\ref{wilf} shows that two patterns $23154$ and $21534$ are equivalent to each other. Computing the first ten cluster numbers and inverting the corresponding series, we get the first ten entries $1$, $1$, $2$, $6$, $24$, $119$, $708$, $4914$, $38976$, $347776$ of the sequence counting permutations that avoid either of them. 
\end{example}

\subsubsection{Clusters and chains for patterns of length~$4$}

Let us now consider patterns of length~$4$. The equivalence classes of these are as follows (see~\cite{Elizalde}):
\begin{enumerate}
 \item[I.] $1234\simeq4321$
 \item[II.] $2413\simeq3142$
 \item[III.] $2143\simeq3412$
 \item[IV.] $1324\simeq4231$
 \item[V.] $1423\simeq3241\simeq4132\simeq2314$
 \item[VI.] $1342\simeq2431\simeq4213\simeq3124\simeq1432\simeq2341\simeq4123\simeq3214$
 \item[VII.] $1243\simeq3421\simeq4321\simeq2134$ 
\end{enumerate}

The case~I will be considered in Section~\ref{app-chain}. In each of the cases VI and VII, the pattern has no self-overlaps, so Corollary~\ref{1non} applies.  

A very special feature of all patterns of length~$4$ (except for the case~I) is that they only have self-overlaps of lengths~$1$ and~$2$, so patterns in a cluster overlap only if they are neighbours. Moreover, in this case chains actually coincide with clusters. Let us explain that. In fact, we shall show that even in the case of self-overlapping patterns, every labelling of an unlabelled chain that is compatible with ordering of each of the patterns gives a genuine chain. Let us show that for the pattern~$1324$, in other cases the proof is similar. Assume that there is a labelling of some unlabelled $m$-chain $c$ for which its proper beginning is a linkage of $m$ patterns as well, and that $m$ is the smallest possible integer for which it happens. Then, clearly, the first $(m-1)$ patterns in both $c$ and its beginning are the same, and the linkages of the $m^\text{th}$ pattern with the $(m-1)^\text{st}$ one differ. However, this would mean that the following unlabelled chain has a consistent labelling:
 $$
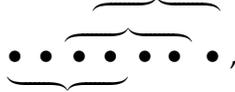
\underbrace{\bullet\,\,\bullet\,
\makebox[0pt][l]{$\overbrace{\phantom{\bullet\,\,\,\bullet\,\,\,\bullet\,\,\,\bullet\,\,}}$}\bullet\makebox[0pt][l]{$\overbrace{\vphantom{\overbrace{\phantom{\bullet\,\,\bullet\,\,\bullet\,\,\bullet\,\,}}}\phantom{\bullet\,\,\bullet\,\,\bullet\,\,\bullet\,\,}\bullet}$\,\,,}\,\,\bullet\,\,}\,\bullet\,\,\,\bullet
 $$ 
but for each labelling the orders of fourth and the fifth bullet coming from the second and the third bullet pattern contradict one another, which is impossible. 

\subsubsection{Case of the pattern~$1324$}

\begin{theorem}\label{1324}
The cluster numbers $c_{n,l}=\cl^{1324}_{n,l}$ for $1324$ satisfy the recurrence relations
\begin{equation}\label{rec-catalan}
c_{n,l}=\sum_{4\le 2k+2\le n}\frac{1}{k+1}\binom{2k}{k}c_{n-2k-1,l-k}  
\end{equation}
with initial conditions $c_{1,l}=\delta_{0,l}$, $c_{2,l}=0$, $c_{3,l}=0$.
Consequently, the generating function for occurrences of $1324$ is 
 $$
\left(1-x-\sum_{n\ge2,l\ge1}\frac{c_{n,l}x^n(t-1)^l}{n!}\right)^{-1}.
 $$
\end{theorem}

\begin{proof}
Counting clusters is reduced to counting total orderings of the corresponding posets. Let us assume that the first $k+1$ patterns have two-element overlaps, and the following overlap involves just one element. For a cluster with $\sigma=a_1a_2\ldots a_{2k+1}a_{2k+2}a_{2k+3}\ldots$, this means that
\begin{equation}\label{catalan}
a_1<a_3<a_2<a_4, a_3<a_5<a_4<a_6, \ldots, a_{2k-1}<a_{2k+1}<a_{2k}<a_{2k+2}, 
\end{equation}
that $\{a_1,\ldots,a_{2k+2}\}=\{1,\ldots,2k+2\}$, and that $\st(a_{2k+2}a_{2k+3}\ldots)$ is an $(l-k)$-cluster. To prove~\eqref{rec-catalan}, we notice that the number of permutations $a_1a_2\ldots a_{2k+2}$ of $\{1,\ldots,2k+2\}$ for which the conditions~\eqref{catalan} are satisfied is given by the number of standard Young tableaux of size $2\times k$:  clearly, $a_1=1$, $a_{2k+2}=2k+2$, and 
 $$
a_2,a_3,a_4,\ldots,a_{2k+1} \quad\leftrightarrow\quad  
\begin{tabular}{|p{0.8cm}|p{0.8cm}|p{0.8cm}|p{0.8cm}|p{0.8cm}|}
\hline
$a_3$& $a_5$ &$a_7$ &\ldots&$a_{2k+1}$\\
\hline
$a_2$& $a_4$& $a_6$ &\ldots&$a_{2k}$\\
\hline
\end{tabular}
 $$
gives a bijection with standard Young tableaux. The number of such tableaux is equal to the Catalan number $\frac{1}{k+1}\binom{2k}{k}$ (see, for example~\cite{Stan}), and the recurrence relation~\eqref{rec-catalan} follows.
\end{proof}

\begin{example}
Computing the first ten cluster numbers and inverting the corresponding series, we get the first ten entries $1$, $1$, $2$, $6$, $23$, $110$, $632$, $4229$, $32337$, $278204$ of the sequence which is indeed counting permutations that avoid~$1324$ (A113228 in~\cite{oeis}). 
\end{example}

\subsubsection{Case of the pattern~$1423$}

\begin{theorem}\label{1423}
The cluster numbers $c_{n,l}=\cl^{1423}_{n,l}$ for $1423$ satisfy the recurrence relations
\begin{equation}\label{shuffling-rec}
c_{n,l}=\sum_{4\le 2k+2\le n}\binom{n-k-2}{k}c_{n-2k-1,l-k} 
\end{equation}
with initial conditions $c_{1,l}=\delta_{0,l}$, $c_{2,l}=0$, $c_{3,l}=0$. Consequently, 
the generating function for occurrences of $1423$ is 
 $$
\left(1-x-\sum_{n\ge2,l\ge1}\frac{c_{n,l}x^n(t-1)^l}{n!}\right)^{-1}.
 $$
\end{theorem}

\begin{proof}
Similarly to the proof of Theorem~\ref{1324}, counting clusters is reduced to counting total orderings of the corresponding posets.
Let us assume that the first $k+1$ patterns have two-element overlaps, and the following overlap involves just one element. For a cluster with $\sigma=a_1a_2\ldots a_{2k+1}a_{2k+2}a_{2k+3}\ldots$, this means that
\begin{equation}\label{shuffling1}
a_1<a_3<a_4<a_2, a_3<a_5<a_6<a_4, \ldots, a_{2k-1}<a_{2k+1}<a_{2k+2}<a_{2k}, 
\end{equation}
so 
\begin{equation}\label{shuffling}
a_1<a_3<\ldots<a_{2k-1}<a_{2k+1}<a_{2k+2}<a_{2k}<\ldots<a_4<a_2, 
\end{equation}
$\{a_1,a_3,\ldots,a_{2k+1}\}=\{1,2,\ldots,k+1\}$, $a_{2k+2}=k+2$,  and $\st(a_{2k+2}a_{2k+3}\ldots)$ is an $(l-k)$-cluster. To prove~\eqref{shuffling-rec}, we notice that the number of way to distribute numbers between the increasing sequence~\eqref{shuffling} and the $(l-k)$-chain $\st(a_{2k+2}a_{2k+3}\ldots)$ is equal to the number of way to choose the $k$ numbers  
$a_{2k},\ldots,a_2$. The latter is clearly the binomial coefficient $\binom{n-k-2}{k}$, and the recurrence relation~\eqref{shuffling-rec} follows.
\end{proof}

\begin{example}
Computing the first ten cluster numbers and inverting the corresponding series, we get the first ten entries $1$, $1$, $2$, $6$, $23$, $110$, $631$, $4218$, $32221$, $276896$ of the sequence counting permutations that avoid~$1423$. 
\end{example}

\subsubsection{Case of the pattern~$2143$}

\begin{theorem}\label{2143}
The cluster numbers $c_{n,l}=\cl^{2143}_{n,l}$ for $2143$ are given by the formula
 $$
 c_{n,l}=\sum_{2\le p<n-2}c_{n,l}(p),
 $$
where the numbers $c_{n,l}(p)$ satisfy the recurrence relations
\begin{equation}\label{2143-rec}
c_{n,l}(p)=\sum_{4\le2k+2\le q\le n}\binom{q-p-1}{2k-2}(p-1)(n-q)c_{n-2k-1,l-k}(q-2k).  
\end{equation}
with initial conditions $c_{1,l}(p)=\delta_{0,l}\delta_{1,p}$, $c_{2,l}(p)=0$, $c_{3,l}(p)=0$. 
Consequently, the generating function for occurrences of $2143$ is 
 $$
\left(1-x-\sum_{n\ge2,l\ge1}\frac{c_{n,l}x^n(t-1)^l}{n!}\right)^{-1}.
 $$
\end{theorem}

\begin{proof}
Similarly to the proof of Theorem~\ref{1324}, counting clusters is reduced to counting total orderings of the corresponding posets. Let $c_{n,l}(p)$ be the number of $l$-clusters for which $\sigma$ is a permutation of length~$n$ with $\sigma(1)=p$.
Let us assume that the first $k+1$ patterns in our cluster have two-element overlaps, and the following overlap involves just one element. For a cluster with $\sigma=a_1a_2\ldots a_{2k+1}a_{2k+2}a_{2k+3}\ldots$, this means that
\begin{equation}\label{monotone1}
a_2<a_1<a_4<a_3, a_4<a_3<a_6<a_5, \ldots, a_{2k}<a_{2k-1}<a_{2k+2}<a_{2k+1}, 
\end{equation}
so 
\begin{equation}\label{monotone}
a_2<a_1<a_4<a_3\ldots<a_{2k}<a_{2k-1}<a_{2k+2}<a_{2k+1}, 
\end{equation}
and $\st(a_{2k+2}a_{2k+3}\ldots)$ is an $(l-k)$-cluster. Assume that $a_1=p$. To prove~\eqref{2143-rec}, we notice that if $a_{2k+2}=q$, then there are $\binom{q-p-1}{2k-2}$ ways to pick the numbers $a_3,\ldots,a_{2k}$, $p-1$ ways to pick~$a_2$, $(n-q)$ ways to pick $a_{2k+1}$, and $c_{n-2k-1,l-k}$ ways to pick the remaining $(l-k)$-cluster (where the entry $q$ is $(q-2k)^\text{th}$ biggest). This completes the proof.
\end{proof}

\begin{example}
Computing the first ten cluster numbers and inverting the corresponding series, we get the first ten entries $1$, $1$, $2$, $6$, $23$, $110$, $631$, $4223$, $32301$, $277962$ of the sequence counting permutations that avoid~$2143$. 
\end{example}

In the last remaining case (II in the list above), we have no trick like above that would simplify the computations, so we shall use the most general strategy for chain enumeration, which allows to compute the chain numbers rather fast (polynomially in~$n$) for all sets of forbidden patterns. There is an obvious similarity with the approach of Kitaev and Mansour in~\cite{KitManMult}. 

\subsubsection{Case of the pattern~$2413$}

\begin{theorem}\label{2413}
The cluster numbers $c_{n,l}=\cl^{2413}_{n,l}$ for $2413$ are given by the formula
 $$
 c_{n,l}=\sum_{1<p<q-1<n}c_{n,l}(p,q),
 $$
where the numbers $c_{n,l}(p,q)$ satisfy the recurrence relations
\begin{multline}\label{2413-rec}
c_{n,l}(p,q)=\sum_{r<p<s<q}c_{n-2,l-1}(r,s-1)+\\+\sum_{p<r<s<q}(p-1)c_{n-3,l-1}(r-1,s-1)+\sum_{p<r<q<s}(p-1)c_{n-3,l-1}(r-1,s-2).    
\end{multline}
with initial conditions $c_{2,l}(p,q)=0$, $c_{3,l}(p,q)=0$, $c_{4,l}(p,q)=\delta_{l,1}\delta_{p,2}\delta_{q,4}$. 
Consequently, the generating function for occurrences of $2143$ is 
 $$
\left(1-x-\sum_{n\ge2,l\ge1}\frac{c_{n,l}x^n(t-1)^l}{n!}\right)^{-1}.
 $$
\end{theorem}

\begin{proof}
This statement is straightforward. Indeed, let us consider an $n$-cluster with $\sigma=a_1a_2a_3\ldots$. The first pattern in that cluster intersects with its neighbour by either two or one elements. In the first case, we have $a_3<a_2<a_4<a_1$, so if we fix $a_1$ and $a_2$, and forget about them, we are left with an $(n-1)$-cluster, and we should sum over all choices of $a_3$ and $a_4$ for its first entries. If, on the contrary the first overlap uses just one element, then there are $(a_1-1)$ choices for $a_3$, and we should distinguish between the cases $a_5>a_2$ and $a_5<a_2$: in the first case $a_5$ is the $(a_5-1)^\text{st}$ biggest in the remaining cluster, while in the second case it is the $(a_5-2)^\text{nd}$ biggest.
\end{proof}

\begin{example}
Computing the first ten cluster numbers and inverting the corresponding series, we get the first ten entries $1$, $1$, $2$, $6$, $23$, $110$, $632$, $4237$, $32465$, $279828$ of the sequence counting permutations that avoid~$2413$. 
\end{example}

\subsubsection{Case of two patterns~$\{132,231\}$}

The following theorem is mentioned in~\cite{Elizalde}.

\begin{theorem}\label{132and231}
The cluster number $c_{n,l}=\cl^{132,231}_{n,l}$  is not equal to zero only for $n=2l+1$, and in this case is equal to $E_{2l+1}$, the tangent number~\cite{Stan}, so the generating function for occurrences of $\{132,231\}$ is 
\begin{equation}\label{tan}
\left(1-\frac{\tanh{x\sqrt{1-t}}}{\sqrt{1-t}}\right)^{-1}. 
\end{equation}
\end{theorem}

\begin{proof}
This pair of patterns has no self-overlaps at all (both for a pattern with itself, and two patterns with each other); clearly, clusters are nothing but ``up--down'' permutations, that is permutations $a_1a_2\ldots a_{2l}a_{2l+1}$
for which $a_1<a_2>a_3<a_4>\ldots<a_{2l}>a_{2l+1}$. It is well known that the number of such permutations is equal to the tangent number. The square roots in~\eqref{tan} account for the fact that every up--down permutation of length~$2l+1$ is an $l$-cluster.
\end{proof}

\subsection{Applications of the chain method}\label{app-chain}

The cases we consider in this section are some of those where the number of $k$-chains is substantially smaller than the number of $k$-clusters. For example, this happens if the set of forbidden patterns contains the pattern $12\ldots a$, which marks rises of length~$a$ in permutations. 

\subsubsection{Case of the pattern~$12\ldots a$}

The following result is well known.
\begin{theorem}[\cite{EN,GJ1,KitPOP}]\label{rise}
The multiplicative inverse of the exponential generating function for patterns avoiding $12\ldots a$ (``permutations without $a$ consecutive rises'') is given by the formula
\begin{equation}\label{GJ1}
\sum_{k\ge0}\frac{x^{ka}}{(ka)!}-\sum_{k\ge0}\frac{x^{ka+1}}{(ka+1)!}.
\end{equation} 
\end{theorem}

\begin{proof}
Indeed, $k$-chains for $k\ge2$ are as follows:
\begin{itemize}
 \item[-] the only $2$-chain is $[12\ldots a]$;
 \item[-] the only $3$-chain is $[12\ldots(a+1)]$;
 \item[-] the only $4$-chain is $[12\ldots(2a)]$;
 \item[-] the only $5$-chain is $[12\ldots(2a+1)]$;
 \item[-] \ldots
 \item[-] the only $(2k)$-chain is $[12\ldots(ka)]$;
 \item[-] the only $(2k+1)$-chain is $[12\ldots(ka+1)]$;
 \item[-] \ldots
\end{itemize}
\end{proof}

Our next result classifies chains for $\{123,132\}$. It is interesting to compare it with the result of Claesson~\cite{Claesson} stating that the number of permutations of length~$n$ avoiding~$\{123,132\}$ is equal to the number of involutions of length~$n$. 

\subsubsection{Case of two patterns~$\{123,132\}$}

\begin{theorem}\label{invol}
The numbers $c_{n,l}$ of $l$-chains of length~$n$ for $\{123,132\}$ satisfy the recurrence relations
\begin{equation}\label{invol-rec}
c_{n,l}=(1-\delta_{n\bmod{3},2})+\sum_{3\le 3k\le n}(n-3k+1)c_{n-3k+1,l-2k+1}+\sum_{4\le 3k+1\le n}(n-3k)c_{n-3k,l-2k}
\end{equation}
with initial conditions $c_{1,l}=\delta_{1,l}$, $c_{2,l}=0$. Consequently, 
the exponential generating function for permutations avoiding $\{123,132\}$ is 
 $$
\left(1-x+\sum_{n\ge2,l\ge2}\frac{(-1)^lc_{n,l}x^n}{n!}\right)^{-1}
 $$
\end{theorem}

\begin{proof}
The permutation~$132$ has no self-overlaps with itself, and have one nontrivial overlap with~$123$. It is easy to see that every $l$-chain for this pair of permutations is either an $l$-chain for $123$ or, for some $k\le l$ it starts from a $(k-1)$-chain for $123$ which overlaps with $132$ (by one or two entries, depending on the parity of~$k$, as for the single pattern $12\ldots a$ in Theorem~\ref{rise}), which overlaps by its last entry with an $(l-k)$-chain. The only parameter here that varies is the entry in the place of~$3$ for~$132$. Taking that into account, we obtain the above recurrence relation. In it, the first summand counts chains for $123$, while the two sums take care of the segment to the first occurrence of $132$ for both possible parities.
\end{proof}

\begin{example}
Computing the first ten chain numbers and inverting the corresponding series, we recover the first ten entries $1$, $1$, $2$, $4$, $10$, $26$, $76$, $232$, $764$, $2620$ of the sequence counting involutions (A000085 in~\cite{oeis}).  
\end{example}

\section{Appendix: a homological interpretation of the main result}\label{shuffle}

In this section, we discuss another interpretation of the formula~\eqref{master}, putting our bijection~$I$ in the context of homological algebra for shuffle algebras~\cite{Ronco} and of modules over the associative operad~\cite{Lat1,Lat2}. 

\subsection{Shuffle algebras}

A shuffle algebra, as defined in the paper of Maria Ronco~\cite{Ronco}, is a graded vector space $A=\bigoplus_{k\ge0}A_k$ together with maps
 $$
\gamma\colon A_n\otimes A_m\otimes\k Sh(n,m)\to A_{n+m},
 $$
subject to certain associativity condition. Here $Sh(n,m)$ denotes the set of all $(n,m)$-shuffles, that is permutations $\sigma\in S_{n+m}$ for which $\sigma(1)<\ldots<\sigma(n)$ and $\sigma(n+1)<\ldots<\sigma(n+m)$. 

One can slightly re-phrase this definitition. Define a new monoidal structure on graded vector spaces as follows: if $A=\bigoplus_{k\ge 0}A_k$ and $B=\bigoplus_{k\ge0}B_k$ are two graded vector spaces, we put 
\begin{equation}\label{shuffleprod}
(A\boxtimes B)_n=\bigoplus_{i=0}^n A_i\otimes B_{n-i}\otimes\k Sh(i,n-i). 
\end{equation}

This defines a monoidal structure on graded vector spaces, and a shuffle algebra is a monoid in this category. The advantage of this approach is that arbitrary constructions and definitions available for monoidal categories (modules, homology, resolutions) enter our story for free.


Within the framework of shuffle algebras, our results of Section~\ref{chains} should be stated as follows. It is easy to see that the graded vector space $\bigoplus_{n\ge0}\k S_n$ is a shuffle algebra; moreover, this shuffle algebra is a free shuffle algebra with one generator of degree~$1$~\cite{Ronco}. 

Let $A$ be the graded vector space whose $n^\text{th}$ component $A_n$ is spanned by all permutations avoiding permutations from the given set $P$. In other words, $A$ is the quotient of the free shuffle algebra with one generator by the two-sided shuffle ideal generated by all permutations from~$P$. We shall now exhibit a free right module resolution of it trivial right module. Let $C$ be the graded vector space whose $n^\text{th}$ component $C_n$ is spanned by $n$-chains. Using the involution~$I$, one can endow the graded space~$C\boxtimes A$ with a structure of a chain complex. The differential $d$ of this complex maps a basis element $\sigma\mid\lambda$ (in the notation of Section~\ref{chains}) to $0$, if $I$ moves the bar in this element to the right, and to $I(\sigma\mid\lambda)$ if $I$ moves the bar to the left. Clearly, $d^2=0$, and this complex is acyclic in positive degrees, so it is a resolution of the trivial module. Computing Euler characteristics of its graded components results in the equation~\eqref{master}. 

\subsection{Modules over the associative operad}

This last section is intended for those readers whose intuition, as it is for us, comes from the operad theory. Essentially, it re-tells the shuffle algebra approach in a slightly different way, explaining also the place for classical pattern avoidance in the story.

Studying varieties of algebras, that is, algebras satisfying certain identities, goes back to works of Specht~\cite{Specht}. The notions of $T$-ideals and $T$-spaces formalize the ways to derive identities from one another. One natural way to study identities is to define an analogue of a Gr\"obner basis for an ideal of identities. This approach is taken in works of Latyshev \cite{Lat1,Lat2} who suggested a combinatorial approach to study associative algebras with additional identities via standard bases of the corresponding $T$-spaces. His approach can be described as follows. For each ``$T$-space'' (in other words, right ideal in the associative operad), he defines a version of a Gr\"obner basis; such a basis would allow to study arbitrary relations via monomials avoiding certain patterns. Here, for once, by a pattern we mean a classical pattern (its occurrence does not have to be as a consecutive subword). This approach has a slight disadvantage. Namely, even though the actual Gr\"obner bases of relations are expected to be finite (at least, the famous result of Kemer~\cite{Kemer} states that in principle there exists a finite set of generating identities), they are difficult to compute, as there is no algorithm comparable to the one due to Buchberger in the associative algebra case~\cite{Ufn}. Somehow, this trouble disappears if we study left ideals in the associative operad. In terms of divisibility of monomials, left and right ideals lead to two very different combinatorial problems: right divisibility translates into occurrences of classical patterns, whereas left divisibility yields occurrences of consecutive ones! For consecutive patterns, the intuition of \cite{DK,DKRes} for Gr\"obner bases and resolutions applies directly, and our results can be interpreted in terms of appropriate free left module resolutions. Let us remark that the difference between chains and clusters is exactly the same as the difference between the Anick's resolution for associative algebras and the not-quite-minimal resolution constructed in~\cite{DKRes}; in this framework we managed to obtain an explicit canonical minimal resolution, unlike what happens in the general operadic case. 

\bibliographystyle{amsplain}
\bibliography{patterns}

\providecommand{\bysame}{\leavevmode\hbox to3em{\hrulefill}\thinspace}
\providecommand{\MR}{\relax\ifhmode\unskip\space\fi MR }
\providecommand{\MRhref}[2]{%
  \href{http://www.ams.org/mathscinet-getitem?mr=#1}{#2}
}
\providecommand{\href}[2]{#2}
\begin{thebibliography}{10}

\bibitem{Anick}
David~J. Anick, \emph{On the homology of associative algebras}, Trans. Amer.
  Math. Soc. \textbf{296} (1986), no.~2, 641--659.

\bibitem{Claesson}
Anders Claesson, \emph{Generalized pattern avoidance}, European J. Combin.
  \textbf{22} (2001), no.~7, 961--971.

\bibitem{DKRes}
Vladimir Dotsenko and Anton Khoroshkin, \emph{Free resolutions via {G}r\"obner
  bases}, Preprint~\texttt{arXiv:0912.4895}.

\bibitem{DK}
\bysame, \emph{Gr{\"o}bner bases for operads}, To appear in {Duke Math. J.}

\bibitem{DVJ}
Vladimir Dotsenko and Mikael {Vejdemo Johansson}, \emph{Implementing
  {G}r\"obner bases for operads}, Preprint~\texttt{arXiv:0909.4950}.

\bibitem{Elizalde}
Sergi Elizalde, \emph{Consecutive patterns and statistics on restricted
  permutations}, Ph.D. thesis, Universitat Polit\`ecnica de Catalunya, 2004.

\bibitem{EN}
Sergi Elizalde and Marc Noy, \emph{Consecutive patterns in permutations}, Adv.
  in Appl. Math. \textbf{30} (2003), no.~1-2, 110--125, Formal power series and
  algebraic combinatorics (Scottsdale, AZ, 2001).

\bibitem{GJ2}
Ian~P. Goulden and David~M. Jackson, \emph{An inversion theorem for cluster
  decompositions of sequences with distinguished subsequences}, J. London Math.
  Soc. (2) \textbf{20} (1979), no.~3, 567--576.

\bibitem{GJ1}
\bysame, \emph{Combinatorial enumeration}, John Wiley \& Sons Inc., New York,
  1983, With a foreword by Gian-Carlo Rota, Wiley-Interscience Series in
  Discrete Mathematics.

\bibitem{Kemer}
Alexander~R. Kemer, \emph{Solution of the problem as to whether associative
  algebras have a finite basis of identities}, Soviet Math. Dokl. \textbf{37}
  (1988), no.~1, 60--64.

\bibitem{KitPOP}
Sergey Kitaev, \emph{Partially ordered generalized patterns}, Discrete Math.
  \textbf{298} (2005), no.~1-3, 212--229.

\bibitem{KitHist}
Sergey Kitaev and Toufik Mansour, \emph{A survey of certain pattern problems},
  Preprint, 2003.

\bibitem{KitManMult}
\bysame, \emph{On multi-avoidance of generalized patterns}, Ars Combin.
  \textbf{76} (2005), 321--350.

\bibitem{Lat1}
Victor~N. Latyshev, \emph{A general version of standard basis and its
  application to {$T$}-ideals}, Acta Appl. Math. \textbf{85} (2005), no.~1-3,
  219--223.

\bibitem{Lat2}
\bysame, \emph{Combinatorial generators of multilinear polynomial identities},
  J. Math. Sci. (N. Y.) \textbf{149} (2008), no.~2, 1107--1112.

\bibitem{Ma}
Toufik Mansour, \emph{Pattern avoidance in coloured permutations}, S\'em.
  Lothar. Combin. \textbf{46} (2001/02), Art. B46g, 12 pp. (electronic).

\bibitem{MR}
Anthony Mendes and Jeffrey Remmel, \emph{Permutations and words counted by
  consecutive patterns}, Adv. in Appl. Math. \textbf{37} (2006), no.~4,
  443--480.

\bibitem{NZ}
John Noonan and Doron Zeilberger, \emph{The {G}oulden-{J}ackson cluster method:
  extensions, applications and implementations}, J. Differ. Equations Appl.
  \textbf{5} (1999), no.~4-5, 355--377.

\bibitem{Ronco}
Maria Ronco, \emph{Shuffle bialgebras}, Preprint \texttt{arXiv:math/0703437}.

\bibitem{oeis}
Neil J.~A. Sloane, \emph{On-line encyclopedia of integer sequences}, Available
  via the URL \texttt{http://www.research.att.com/$\sim$njas/sequences/}.

\bibitem{Specht}
Wilhelm Specht, \emph{Gesetze in {R}ingen, {I}}, Math. Z. \textbf{52} (1950),
  557--589.

\bibitem{Stan}
Richard~P. Stanley, \emph{Enumerative combinatorics. {V}ol. 2}, Cambridge
  Studies in Advanced Mathematics, vol.~62, Cambridge University Press,
  Cambridge, 1999, With a foreword by Gian-Carlo Rota and appendix 1 by Sergey
  Fomin.

\bibitem{Stein}
Einar Steingr\'imsson, \emph{Generalized permutation patterns -- a short
  survey}, Permutation Patterns, St Andrews 2007, LMS Lecture Note Series, vol.
  376, Cambridge University Press, Cambridge, 2010.

\bibitem{Ufn}
Victor~A. Ufnarovskij, \emph{Combinatorial and asymptotic methods in algebra},
  Algebra, {VI}, Encyclopaedia Math. Sci., vol.~57, Springer, Berlin, 1995,
  pp.~1--196.

\bibitem{Wilf}
Herbert~S. Wilf, \emph{The patterns of permutations}, Discrete Math.
  \textbf{257} (2002), no.~2-3, 575--583.

\end{thebibliography}

\end{document}